\documentclass[twoside,11pt]{article}%
\usepackage{amssymb}
\usepackage{bm}
\usepackage{latexsym}
\usepackage{amsmath,amssymb}
\usepackage{amsthm,enumerate,verbatim}
\usepackage{amsfonts}
\usepackage{amsmath}
\usepackage{multirow}
\usepackage{graphicx}%
\usepackage{graphicx}
\usepackage{subfigure}
\numberwithin{subfigure}{section}
\usepackage{pdftricks}
\usepackage{pst-all}
\usepackage{pgf}
\usepackage{CJK}
\usepackage{graphicx}
\usepackage{caption}
\usepackage{subfigure}

\providecommand{\U}[1]{\protect\rule{.1in}{.1in}}
\providecommand{\U}[1]{\protect\rule{.1in}{.1in}}

\newcommand{\IR}{{\mathbb{R}}}

\newcommand{\IH}{{\mathbb{H}}}
\newcommand{\IB}{{\mathbb{B}}}
\newcommand{\BE}{\begin{equation}}
\newcommand{\EE}{\end{equation}}

\numberwithin{equation}{section}
\newtheorem{definition}{Definition}[section]
\newtheorem{theorem}{Theorem}[section]
\newtheorem{lemma}{Lemma}[section]

\newtheorem{example}{Example}[section]
\newtheorem{proposition}{Proposition}[section]

\newtheorem{corollary}{Corollary}[section]

\begin{CJK}{GBK}{snt}
\end{CJK}

\begin{document}

\begin{CJK*}{GBK}{}

\title{\textbf{On Some Basic Results Related to Affine Functions on Riemmanian Manifolds}}

\date{}
\author{Xiangmei Wang\thanks{College of Science, Guizhou University, Guiyang 550025, P. R. China
(sci.xmwang@gzu.edu.cn).} \and Chong Li\thanks{Department of
Mathematics, Zhejiang University, Hangzhou 310027, P. R. China
(cli@zju.edu.cn). This author was supported in part by the National
Natural Science Foundation of China (grant 11171300) and by Zhejiang
Provincial Natural Science Foundation of China (grant LY13A010011).} \and Jen-Chih Yao\thanks{Center for General
Education, Kaohsiung Medical University,  Kaohsiung 80702, Taiwan  (yaojc@kmu.edu.tw). Research of
this author was partially supported by the National Science Council
of Taiwan under grant NSC  99-2115-M-037-002-MY3.}}

\maketitle

{\noindent\textbf{Abstract.}} \textit{We study some basic properties of the function $f_0:M\rightarrow\IR$ on Hadamard manifolds  defined by
$$
f_0(x):=\langle u_0,\exp_{x_0}^{-1}x\rangle\quad\mbox{for any $x\in M$}.
$$
A characterization for the function to be linear affine is given and a counterexample on Poincar\'{e} plane is provided, which in particular, shows   that   assertions (i) and (ii) claimed in \cite[Proposition 3.4]{Papa2009} are  not true, and that   the function $f_0$ is indeed not quasi-convex. Furthermore, we discuss the convexity properties of the sub-level sets of the function on Riemannian manifolds with constant sectional curvatures.}

\bigskip
{\noindent\textbf{Keywords.}} \textit{Riemannian manifold; Hadamard
manifold; sectional curvature; convex function; quasiconvex
function; linear affine function}

\section{Introduction}\ \
Let $M$ be a Hadamard manifold and let  $x\in M$. Let $T_{x}M$ stand for the tangent space at $x$ to $M$ with the Riemannian scalar product denoted by $\langle \cdot,\cdot\rangle_x$ and let $TM:=\cup_{x\in M}T_xM$. We use $\exp_{x}$ and $P_{x,x_0}$, where $x_0\in M$,  to denote the exponential map of $M$ at $x$ and the
parallel transport from $x_0$ to $x$ (along the unique geodesic joining $x_0$ to $x$), respectively.  Now fix $x_0\in M$ and $u_0\in T_{x_0}M\setminus\{0\}$. Consider the
 vector field $X_0:M\to TM$ and the function $f_0:M\rightarrow\IR$ defined by
\begin{equation}\label{vector-field}
X_0(x):=P_{x,x_0}u_0\quad\quad\mbox{for any $x\in M$}
\end{equation}
and
\begin{equation}\label{form}
f_0(x):=\langle u_0,\exp_{x_0}^{-1}x\rangle\quad\mbox{for any $x\in M$},
\end{equation}
respectively.
Let ${\rm grad}f_0$ denote the gradient of $f_0$. Assertions {\bf(a)} and  {\bf(b)} below  were given in \cite[Proposition 3.4]{Papa2009} (without the  proof for {\bf(b)}).

{\bf(a)} ${\rm grad}f_0=X_0$.

{\bf(b)} $f_0$ is   linear affine   on $M$.\\
Recently, assertions  {\bf(a)} and  {\bf(b)} have been used in  \cite{  Papa2009,Papa2012}
     to study  the proximal point algorithm for quasiconvex/convex functions with Bregman distances on Hadamard manifolds; while assertion {\bf(b)}
     was also used in \cite{Colao2012,  ZhouH2013} to establish some existence results of solutions  for Equilibrium problems and vector optimization problems on Hadamard manifolds, respectively.
  However, assertion {\bf(b)} is clearly not true in general because, by \cite[p. 299, Theorem 2.1]{Udriste1994}),
     any twice differentiable linear affine function on Poincar\'{e} plane $\IH$ (a two dimensional Hadamard manifold of constant curvature $-1$)  is constant.
     Indeed, it has been further shown in \cite[Theorem 2.1]{KLLN2015} that
       assertion (b) is true for any $x_0\in M$ and  $u_0\in T_{x_0}M$ if and only if  $M$ is isometric to the  Euclidean space $\IR^n$. Furthermore, one can easily check that the function $f_0$ defined by \eqref{form} is even not convex, in general, because, otherwise, one has that both $f_0$ and $-f_0$ are convex (and so linear affine). This motivates us to consider the following problem:

{\bf Problem 1} Is the  function $f_0$ defined by \eqref{form}   quasi-convex?

Let $\nabla$ denote the Riemannian connection on $M$ and let $\mathcal{X}(M)$ denote all $C^\infty$ vector field on $M$.
Recall from \cite[P.83]{Udriste1994} that a smooth function  $f:M\rightarrow\IR$ is linear affine if and only if
$$\nabla_X{\rm grad}\,f=0\quad\mbox{for any $X\in\mathcal{X}(M)$}.$$
Specializing in the function $f_0$ defined by \eqref{form}, one is motivated to consider the following problems:

{\bf Problem 2} Is assertion {\bf(a)} true?

{\bf Problem 3} Does the vector field $X_0$ defined by \eqref{vector-field} satisfy
\begin{equation}\label{H-eq-0}
\nabla_XX_0=0\quad\mbox{for any $XY\in\mathcal{X}(M)$}?
\end{equation}

The first purpose of this   paper is to present a characterization in Hadamard manifolds for {\bf(b)} to be true in terms of assertion {\bf(a)} and the parallel transports, and to provide a counterexample on Poincar\'{e} plane to illustrate that the answer to each of Problems 1-3 is negative.
In particular for Problem 2, we   show that the vector field $X_0$ defined by  \eqref{vector-field} is even not a gradient field.
%

Our second purpose in the present paper is, in  spirit of the negative answer to Problem 1, to study   the convexity issue of  sub-level sets of the function $f_0$ defined by \eqref{form} in Riemannian manifolds with constant sectional curvatures. Our main results provide the exact estimate of the constant $c$ such that the sub-level set $L_{c,f_0}:=\{x\in M:f_0(x)\le c\}$ is strongly convex, which in particular improves and extends  the corresponding result in \cite[Corollary 3.1]{Ferreira2005}.

The paper is organized as follows. We review, in Section 2, some
basic notions, notations and some classical results of Riemannian
geometry that will be needed afterward.  The characterization in Hadamard manifolds for {\bf(b)} to be true  and the counterexample on Poincar\'{e} plane  are presented in Section 3. Finally, in Section 4,
 the convexity properties of the sub-level sets of the functions defined by \eqref{form} in Riemannian manifolds with constant sectional curvatures are discussed.

\section{Notations, notions and preliminaries}
\label{sec:2}
In present section, we present some basic notations, definitions and
properties of Riemannian manifolds. The readers are referred to some textbooks for details, for example, \cite{Carmo1992,Sakai1996,Udriste1994}.

Let $M$ be a connected $n$-dimensional Riemannian manifold with the Levi-Civita connection $\nabla$ on $M$. We denote the
tangent space at $x\in M$ by $T_{x}M$ and Let $\mathcal{X}(M)$ denote all ($C^\infty$) vector fields on $M$. By
$\langle\cdot,\cdot\rangle_{x}$ and $\|\cdot\|_x$ we mean the
corresponding Riemannian scalar product and the norm, respectively
(where the subscript $x$ is sometimes omitted). For $x,y\in{M}$, let
$\gamma:[0,1]\rightarrow M$ be a piecewise smooth curve joining $x$
to $y$. Then, the arc-length of $\gamma$ is defined by
$l(\gamma):=\int_{0}^{1}\|\dot{\gamma}(t)\|{\rm d}t$, while the Riemannian
distance from $x$ to $y$ is defined by ${\rm
d}(x,y):=\inf_{\gamma}l(\gamma)$, where the infimum is taken over
all piecewise smooth curves $\gamma:[0,1]\rightarrow M$ joining $x$
to $y$.  We use $\IB(x,r)$ to
denote the open metric ball at $x$ with radius
$r$, that is,
$$\IB(x,r):=\{y\in M:{\rm d}(x, y)<r\}.$$

For a smooth curve $\gamma$,
if $\dot{\gamma}$ is parallel along itself, then  $\gamma$ is called a
geodesic, that is, a smooth curve $\gamma$ is a geodesic if an only
if $\nabla_{\dot{\gamma}}{\dot{\gamma}}=0$. A geodesic
$\gamma:[0,1]\rightarrow M$ joining $x$ to $y$ is minimal if its
arc-length equals its Riemannian distance between $x$ and $y$. By
the Hopf-Rinow theorem \cite{Carmo1992}, $(M,{\rm d})$ is a complete
metric space, and there is at least one minimal geodesic joining $x$
to $y$. The set of
all geodesics $\gamma:[0,1]\rightarrow M$ with $\gamma(0)=x$ and
$\gamma(1)=y$ is denoted by   $\Gamma_{xy}$, that is
$$\Gamma_{xy}:=\{\gamma:[0,1]\rightarrow M:\;\gamma(0)=x,\, \gamma(1)=y\mbox{ and } \nabla_{\dot{\gamma}}\dot{\gamma}=0\}.
$$

Let $\gamma$ be a geodesic. We use $P_{\gamma,\cdot,\cdot}$ to denote the parallel transport on the tangent
bundle $TM$ (defined below) along $\gamma$ with respect to $\nabla$, which is defined by
\begin{equation}\label{def-parallel-trans}
P_{\gamma,\gamma(b),\gamma(a)}v=X(\gamma(b))\quad\mbox{for all } a,b\in\IR\mbox{ and }v\in T_{\gamma(a)}M,
\end{equation}
where $X$ is the unique vector field satisfying
\begin{equation}\label{para}
X(\gamma(a))=v\quad \mbox{and}\quad \nabla_{\dot{\gamma}}X=0.
\end{equation}
Then, for any $a,b\in \IR$, $P_{\gamma,\gamma(b),\gamma(a)}$ is an isometry from $T_{\gamma(a)}M$ to $T_{\gamma(b)}M$. We will
write $P_{y,x}$ instead of $P_{\gamma,y,x}$ in the case when $\gamma$ is a minimal geodesic joining $x$ to $y$
and no confusion arises.

The exponential map of $M$ at $x\in M$ is
denoted by $\exp_x(\cdot):T_xM\rightarrow M$. For a $C^\infty$ function $f:M\rightarrow\IR$, ${\rm grad} f$ and ${\rm Hess} f$ denote its gradient vector and Hessian, respectively. Let $X,Y\in\mathcal{X}(M)$. The Riemannian connection has the expression in terms of parallel transportation, that is,
\begin{equation}\label{covariant-D}
(\nabla_XY)(x)=\lim_{t\rightarrow0}\frac{1}{t}\{P_{\gamma,\gamma(t),\gamma(0)}Y(\gamma(t))-Y(x)\}\quad\mbox{for any $x\in M$},
\end{equation}
where the curve $\gamma$ with $\gamma(0)=x$ and $\dot{\gamma}(0)=X(x)$ (see, e.g., \cite[p. 29 Exercise 5]{Sakai1996}).

A complete simply connected Riemannian manifold of non-positive
sectional curvature is called a {\sl{Hadamard manifold}}. The
following propositions are well-known about the Hadamard manifolds, see, e.g, \cite[p. 221]{Sakai1996}.

\begin{proposition}\label{had} Suppose that $M$ is a Hadamard manifold. Let $p\in M$. Then,
$\exp_p:T_pM\rightarrow M$ is a diffeomorphism, and for any two
points $p,q\in M$ there exists a unique normal geodesic joining $p$
to $q$, which is in fact a minimal geodesic.
\end{proposition}

The following definition presents the notions of different convexities, where item (a) and (b) are known in \cite{Cheeger1972}; see also
\cite{LiLi2009,Walter1974,Wang2010}.

\begin{definition}\label{convexset}
Let $Q$ be a nonempty subset of the Riemannian manifold $M$. Then, $Q$ is said to be

{\rm (a)} weakly convex if, for any $x,y\in Q$, there is a minimal geodesic of $M$ joining $x$ to $y$
and it is in $Q$;

{\rm (b)} strongly convex if, for any $x,y\in Q$, there is just one
minimal geodesic of $M$ joining $x$ to $y$
and it is in $Q$.
\end{definition}

All convexities in a Hadamard manifold coincide and are simply called the convexity. Let $f:M\rightarrow\overline {\IR}$ be a proper function, and let
${\rm dom}f$ denote its domain, that is, ${\rm dom}f:=\{x\in
M:f(x)\neq\infty\}$. We use $\Gamma^f_{xy}$ to denote the set of all
$\gamma\in\Gamma_{xy}$ such that $\gamma\subseteq{\rm dom}f$. In the following definition, item (a) is known in \cite{LiMWY2011,LiY2012} and item (b) is an extension of the one in \cite[p. 59]{Udriste1994}, which is introduced for the case when ${\rm dom}f$ is totally convex.

\begin{definition}\label{convexfunction}
Let $f:M\rightarrow\overline {\IR}$ be a proper function and suppose that ${\rm
dom}f$ is weakly convex.
Then, $f$ is said to be

{\rm (a)} convex if
$$f\circ\gamma(t)\le(1-t)f(x)+tf(y)\quad\mbox{for all }x,y\in {\rm
dom}f,\;\gamma\in\Gamma_{xy}^f,\;t\in[0,1];$$

{\rm (b)} quasi-convex if
$$f\circ\gamma(t)\le\max\{f(x),f(y)\}\quad\mbox{for all }x,y\in {\rm
dom}f,\;\gamma\in\Gamma_{xy}^f,\;t\in[0,1].$$
\end{definition}

Clearly, for a proper function $f$ with a weakly convex domain, the convexity implies the
quasi-convexity. Fixing $c\in\IR$,  we use  $L_{c,f}$  to denote the sub-level set of $f$ defined by
$$L_{c,f}:=\{x\in M:f(x)\le c\}.$$
The following proposition describe the relationship between the  convexities  of a function $f$ and its sub-level sets.
\begin{proposition}\label{remark}
Let $f:M\rightarrow\overline {\IR}$ be a proper function with weakly convex   domain ${\rm
dom}f$. Then, $f$ is quasi-convex if and only if, for each $c\in \IR$,  the sub-level set $L_{c,f}$ is totally convex with restricted to ${\rm dom}f$ in the sense that
  for any $x,y\in L_{c,f}$,    if  $\gamma \in \Gamma^f_{xy}$ then $\gamma\subseteq L_{c,f}$. In particular,   $f$ is quasi-convex if and only if $L_{c,f}$ is strongly convex  for each $c\in \IR$ in the case when ${\rm
dom}f$ is strongly convex.
\end{proposition}

\begin{proof}
We only consider the case when ${\rm dom}f$ is weakly convex (otherwise when ${\rm dom}f$ is weakly convex, the result is immediate by definition).

Suppose that $f$ is quasi-convex. Take $c\in\IR$. Let $x,y\in L_{c,f}\subseteq{\rm dom}f$ and let $\gamma\in\Gamma_{xy}^f$ (i.e., $\gamma$ is a geodesic joining $x$ to $y$ which is contained in ${\rm dom}f$). Then, $f(x)\le c$ and $f(y)\le c$. Noting that $f$ is quasi-convex, it follows that
$$f\circ\gamma(t)\le\max\{f(x),f(y)\}\le c\quad\mbox{for all }t\in[0,1].$$
This implies that $\gamma\subseteq L_{c,f}$ and so $L_{c,f}$ is totally convex restricted to ${\rm dom}f$ since $x,y\in L_{c,f}$ and $\gamma\in\Gamma_{xy}^f$ are arbitrary.

Conversely, suppose that $L_{c,f}$ is totally convex restricted to ${\rm dom}f$ for each $c\in \IR$. Let $x,y\in {\rm dom}f$ and let $\gamma\in\Gamma_{xy}^f$.
Set $c_0:=\max\{f(x),f(y)\}$. Then, by assumption, $\gamma\subseteq L_{c_0,f}$, that is,
$$f\circ\gamma(t)\le c_0=\max\{f(x),f(y)\}\quad \mbox{for all $t\in[0,1]$}.$$
This implies that $f$ is quasi-convex since $x,y\in {\rm dom}f$ and $\gamma\in\Gamma_{xy}^f$ are arbitrary.
The proof is complete.
\end{proof}

\section{Linear affine functions and counterexample on Hadamard manifolds}
\label{sec:2}

For the whole  section, we assume that $M$ is a Hadamard manifold.
Consider a proper convex function $f:M\rightarrow\overline {\IR}$ on $M$. We define the subdifferential of $f$ at
$x\in{\rm dom}f$ by
$$\partial f(x):=\{v\in T_{x}M:\; f(y)\ge f(x)+\langle v,\gamma'(0)\rangle\;\;\mbox{ for all } y\in {\rm dom}f\mbox{ and }
\gamma\in\Gamma^f_{xy}\}.
$$
By \cite[p. 74]{Udriste1994} (see also \cite[Proposition 6.2]{LiY2012}), $\partial f(x)$ is a
 nonempty, compact and convex set for any $x\in {\rm int}({\rm dom}f)$, where  ${\rm int}Q$ denotes the topological interior of a subset $Q$ of $M$.
 Let $f:M\rightarrow\overline {\IR}$ be a proper function with   convex domain.
 Recall that $f$ is  linear affine if both $f$ and $-f$ are convex.  Furthermore, if $f$ is of $C^2$
 and ${\rm dom}f$ is open, its second covariant differetial ${\rm Hess}f$ is defined by
 $${\rm Hess}f(X,Y)=\langle\nabla_{X}{\rm grad}f,Y\rangle\quad\mbox{for any }X,Y\in\mathcal{X}(M).
 $$
Then $f$ is linear affine if and only if ${\rm Hess}f=0$ on ${\rm dom}f$; see \cite[P.83]{Udriste1994}.
The following theorem present, in particular,  a characterization in Hadamard manifolds for assertion
{\bf(b)} to be true in terms of assertion {\bf(a)} and the parallel transports.

\begin{theorem} \label{affine2}
Let $f:M\rightarrow\overline {\IR}$ be a proper function and suppose that ${\rm dom}f$ is a nonempty open convex subset.  
If function $f$ is linear affine, then, for any $x_0\in{\rm dom}f$, there exists  $u_0\in T_{x_0}M$ such that
\begin{equation}\label{2-picewise}
P_{x,x_0}u_0=P_{x,z}
\circ P_{z,x_0}u_0  \quad\mbox{for any  }  (z,x)\in {\rm dom}f\times {\rm dom}f,
\end{equation}
\begin{equation}\label{function-g}
{\rm grad} f(x)=P_{x,x_0}u_0\quad\mbox{for any } x\in {\rm dom}f
\end{equation}
and
\begin{equation}\label{function}
f(x)=f(x_0)+\langle u_0,\exp_{x_0}^{-1}x\rangle \quad\mbox{for any } x\in {\rm dom}f.
\end{equation}
Conversely, if there exist $x_0\in {\rm dom}f$ and   $u_0\in T_{x_0}M$ such that \eqref{2-picewise} and \eqref{function-g} hold, then
  $f$ is linear affine. 

\end{theorem}

\begin{proof} Assume that    $f$ is linear affine.  Then both $f$ and
$-f$ are convex.   Take $x_0\in{\rm dom}f$ and note that $ {\rm dom}f$ is open.  It follows that both $\partial f(x_0)$ and $\partial (-f(x_0))$
are nonempty. Thus one can chose $u_0\in\partial f(x_0)$ and
$u_0'\in\partial(-f(x_0))$, respectively. Then, by definition, we have
that, for any $x\in {\rm dom}f$, 
\begin{equation}\label{function1}
\begin{array}{ll}
 f(x)\ge f(x_0)+\langle u_0,\exp_{x_0}^{-1}x\rangle \quad\mbox{and}\quad
 -f(x)\ge -f(x_0)+\langle u_0',\exp_{x_0}^{-1}x\rangle ;
\end{array}
\end{equation}
hence $
\langle u_0+u_0',\exp_{x_0}^{-1}x\rangle_{x_0}\le0 $  for any $x\in {\rm dom}f$.
This implies that
$u_0+u_0'=0$, that is $u_0'=-u_0$  (as $ {\rm dom}f$  is open).
Thus \eqref{function} follows from \eqref{function1}. Furthermore, noting that  $f$ is of class $C^\infty$ by \eqref{function}, one then has that  ${\rm Hess}f=0$ on ${\rm dom}f$, that is, $$\mbox{${\rm Hess}f(X,Y)=\langle\nabla_X{\rm grad} f,Y\rangle=0$\;\; for any $X,Y\in \mathcal{X}({\rm dom}f)$}.$$
In particular, one has that
$$
 \nabla_{\dot{\gamma}_{xz}}{\rm grad} f =0 \quad\mbox{for any } x,z\in {\rm dom}f,$$
where  ${\gamma}_{xz}$ is the geodesic joining $x$ and $z$, which   lies in ${\rm dom}f$. This, together with the definition of parallel transport (e.g., \eqref{def-parallel-trans}), implies that
\begin{equation}\label{function40}
{\rm grad} f(x)=P_{x,z}{\rm grad} f(z)\quad\mbox{for any } x,z\in {\rm dom}f.
\end{equation}
Note further that, for any $u\in   T_{x_0}M$, one has
$$\langle{\rm grad}f(x_0),u\rangle_{x_0}=\frac{{\rm
d}}{{\rm d}t}f\circ\exp_{x_0}t u\mid_{t=0}=\langle u_0,u\rangle_{x_0}.$$
It follows that ${\rm grad}f(x_0)=u_0$. This, together with  \eqref{function40}, implies that  \eqref{2-picewise} and \eqref{function-g}  hold.


Now, suppose that \eqref{2-picewise} and \eqref{function-g} hold for some $x_0\in {\rm dom}f$ and $u_0\in T_{x_0}M$. Let $x\in {\rm dom}f$ and $X\in \mathcal{X}({\rm dom}f)$. Let $\gamma:[-\varepsilon,\varepsilon]\rightarrow {\rm dom}f$ be the geodesic contained in ${\rm dom}f$ with $\gamma(0)=x$ and $\dot{\gamma}(0)=X(x)$. Let $t\in[-\varepsilon,\varepsilon]$.
We see from \eqref{function-g} that
$${\rm grad} f(x)=P_{x,x_0}u_0,\;{\rm grad} f(\gamma(t))=P_{\gamma(t),x_0}u_0.$$
In light of \eqref{2-picewise}, it follows that $$P_{x,\gamma(t)}{\rm grad} f(\gamma(t))=P_{x,\gamma(t)}\circ P_{\gamma(t),x_0}u_0=P_{x,x_0}u_0={\rm grad} f(x).$$
Noting that $P_{x,\gamma(t)}=P_{\gamma,x,\gamma(t)}$, one gets by \eqref{covariant-D} that
$$(\nabla_X{\rm grad} f)(x)=\lim_{t\rightarrow0}\frac{1}{t}\{P_{x,\gamma(t)}{\rm grad} f(\gamma(t))-{\rm grad} f(x)\}=0.$$
Since $X\in \mathcal{X}({\rm dom}f)$ and $x\in {\rm dom}f$ are arbitrary, we conclude that ${\rm Hess}f=0$ on ${\rm dom}f$, and so $f$ is linear affine.
The proof is complete.
\end{proof}

The remainder of this section is to construct a counterexample on Poincar\'{e} plane to illustrate that  the answer to each of Problems 1-3 is negative.
To do this, let
$$M=\IH=:\{(t_1,t_2)\in\mathbb{R}^{2}|\ t_2>0\},$$
be the  {Poincar\'{e}
plane} endowed with the Riemannian metric, in terms of the natural coordinate system,   defined by
\begin{equation}\label{metric0}
g_{11}=g_{22}:=\frac{1}{t_2^2},\ g_{12}:=0\text{ for each }
(t_1,t_2)\in \IH.
\end{equation}

The sectional curvature of $\IH$ is equal to $-1$ (see, e.g., \cite[p. 160]{Carmo1992}), and the geodesics
on $\IH$ are the semilines $\gamma(a;\cdot):=(\gamma^1(a;\cdot),\gamma^2(a;\cdot))$ (through $(a,1)$), and the semicircles $\gamma(b,r;\cdot):=(\gamma^1(b,r;\cdot),\gamma^2(b,r;\cdot))$ with center at $(b,r)$ and radius $r$), which admit the following natural
parameterizations:
\begin{equation}\label{geodesic}
\left\{\begin{array}{l}
\gamma^1(a;s)=a\\
  \gamma^2(a;s)=e^{s}\\
  \end{array} \right.
\quad  \mbox{and} \quad
\left\{\begin{array}{l}
 \gamma^1(b,r;s)=b-r\tanh s\\
 \gamma^2(b,r;s)=\frac{r}{\cosh s}\\
\end{array} \right.
  \quad \mbox{for any }  s\in\IR,
 \end{equation}
  respectively;
see e.g., \cite[p. 298]{Udriste1994}.

By \cite[p. 297]{Udriste1994},
  the Riemannian connection $\nabla$ on $\IH$ (in terms of the natural coordinate system) has the components:
\begin{equation}\label{conenction-H}
\Gamma_{11}^1=\Gamma_{22}^1=\Gamma_{12}^2=\Gamma_{21}^2=0,\quad \Gamma_{12}^1=\Gamma_{21}^1=\Gamma_{22}^2=-\frac{1}{t_2}
\quad\mbox{and}\quad \Gamma_{11}^2=\frac{1}{t_2}.
\end{equation}
  Hence, noting the expression of the connection $\nabla$ given in \cite[p. 51]{Carmo1992},
one has the following formular for the connection $\nabla$ on $\IH$:
\begin{equation}\label{conenction-H-XY}
\nabla_YX=\left(Y^1\frac{\partial X^1}{\partial t_1}+Y^2\frac{\partial X^1}{\partial t_2}-\frac{1}{t_2}X^1Y^2-\frac{1}{t_2}X^2Y^1,Y^1\frac{\partial X^2}{\partial t_1}+Y^2\frac{\partial X^2}{\partial t_2}+\frac{1}{t_2}X^1Y^1-\frac{1}{t_2}X^2Y^2\right).
\end{equation}
for any $X:=(X^1,X^2),Y:=(Y^1,Y^2)\in\mathcal{X}(\IH)$, where and in sequel, for a differential function $\phi$ on $\IH$,  $\frac{\partial \phi}{\partial t_1}$ and  $\frac{\partial \phi}{\partial t_2}$  denote   the classical  partial derivatives of $\phi$ in $\IR^2$  with respect to the  first  variable  $t_1$ and the second variable $t_2$, respectively.
 Consider   a differentiable function
  $f:\IH\rightarrow\IR$. Then,    using \eqref{metric0}, one concludes that the gradient vector ${\rm grad} f$  and the differential ${\rm d}f$ of $f$ are respectively given by
\begin{equation}\label{gradient-H}
{\rm grad} f{(x)}=t_2^2\left(\frac{\partial f{(x)}}{\partial t_1}\frac{\partial}{\partial t_1}+\frac{\partial f{(x)}}{\partial t_1}\frac{\partial}{\partial t_2}\right)
\end{equation}
and
\begin{equation}\label{differential-H}
{\rm d}f(x)=\frac{\partial f(x)}{\partial t_1}{\rm d}t_1+\frac{\partial f(x)}{\partial t_2}{\rm d}t_2
\end{equation}
for any $x=(t_1,t_2)\in\IH$; see, e.g., \cite[p. 8]{Udriste1994}. .

For convenience, we also   need the expressions of the exponential map $\exp_{x}^{-1}y$ and the geodesic $\gamma_{xy}$ joining $x$ to $y$, which can be found in \cite{WLYJNCA2014}. To this end, let $x:=(t_1,t_2)$ and $y:=(s_1,s_2)$ be in $ \IH$, and  set
\begin{equation}\label{geodesic2}
b_{xy}:=\frac{(s_1)^2+(s_2)^2-((t_1)^2+(t_2)^2)}{2(s_1-t_1)}\quad\mbox{and}\quad r_{xy}:=\sqrt{(s_1-b_{xy})^2+(s_2)^2}
\end{equation}
if $t_1\neq s_1$. Then one has
\begin{equation}\label{geodesic1}
\exp_{y}^{-1}x=\left \{
\begin{array}{ll}
(0,s_2\ln\frac{t_2}{s_2}),&\text{ if } t_1=s_1,\\
\frac{s_2}{r_{xy}}({\rm artanh}\frac{{b_{xy}-s_1}}{r_{xy}}-{\rm
artanh}\frac{{b_{xy}-t_1}}{r_{xy}})(s_2,b_{xy}-s_1),&\text{ if }
t_1\neq s_1.
\end{array}
\right.
\end{equation}
and  $\gamma_{xy}:=(\gamma_{xy}^1,\gamma_{xy}^2)$  with $\gamma_{xy}^1$ and $\gamma_{xy}^2$ defined respectively by
\begin{equation}\label{geodesic-y-z-1}
\gamma_{xy}^1(s):=\left\{
\begin{array}{ll}
t_1, &\text{ if } t_1=s_1,\\
b_{xy}-r_{xy}\tanh\left((1-s)\cdot{\rm artanh}\frac{{b_{xy}-t_1}}{r_{xy}}+s\cdot{\rm
artanh}\frac{{b_{xy}-s_1}}{r_{xy}}\right),\quad &\text{ if } t_1\neq
s_1,
\end{array}\right.
\end{equation}
and
\begin{equation}\label{geodesic-y-z-2}
\gamma_{xy}^2(s):=\left\{
\begin{array}{ll}
e^{(1-s)\cdot\ln t_2+s\cdot\ln
s_2},&\text{ if } t_1=s_1,\\
\frac{r_{xy}}{\cosh\left((1-s)\cdot{\rm
artanh}\frac{{b_{xy}-t_1}}{r_{xy}}+s\cdot{\rm
artanh}\frac{{b_{xy}-s_1}}{r_{xy}}\right)},\quad &\text{ if }
t_1\neq s_1 ,
\end{array}\right.
\end{equation}
for any $s\in [0,1]$. Now we are ready to present the counterexample.

\begin{example}\label{example}{\rm
Let $x_0:=(0,1)$, and let $u_0:=(0,1)\in T_{x_0}\IH$ be a unit vector. Let $f_0:\IH\rightarrow\IR$ and $X_0:\IH\rightarrow T\IH$
be the function  and the vector field defined by \eqref{form} and \eqref{vector-field}, respectively.
We claim that, for each $x=(t_1,t_2)\in\IH$,
\begin{equation}\label{f-value}
f_0(x)=\left \{
\begin{array}{ll}
\ln{t_2},&\text{ if } t_1=0,\\
\frac{b_x}{r_x}\left({\rm artanh}\frac{b_x}{r_x}-{\rm
artanh}\frac{b_x-t_1}{r_x}\right),&\text{ if }
t_1\neq 0,
\end{array}
\right.
\end{equation}
and
\begin{equation}\label{vector-field-paral}
X_0(x)=\left \{
\begin{array}{ll}
(0,t_2),&\text{ if } t_1=0,\\
\left(\frac{b_xt_2^2-t_2(b_x-t_1)}{b_x^2+1},\frac{b_xt_2(b_x-t_1)+t_2^2}{b_x^2+1}\right),&\text{ if }
t_1\neq 0,
\end{array}
\right.
\end{equation}
where, for any $x$ with   $t_1\neq0$,
\begin{equation}\label{b-r-value}
b_x:=b_{xx_0}=\frac{t_1^2+t_2^2-1}{2t_1}\quad\mbox{and}\quad r_x:=r_{xx_0}=\sqrt{b_x^2+1}.
\end{equation}
Indeed, let $x=(t_1,t_2)\in\IH$. Then by \eqref{geodesic1}, we get that
$$\exp_{x_0}^{-1}x=\left \{
\begin{array}{ll}
(0,\ln{t_2}),&\text{ if } t_1=0,\\
\frac{1}{r_x}({\rm artanh}\frac{{b_x}}{r_x}-{\rm
artanh}\frac{{b_x-t_1}}{r_x})(1,b_x),&\text{ if }
t_1\neq 0;
\end{array}
\right.$$
thus \eqref{f-value} follows immediately from definition. To check  \eqref{vector-field-paral},
let $\gamma$ be the geodesic through $x$ and $x_0$. By the definition of $X_0$ and thanks to \eqref{para},
we have to show $\nabla_{\dot{\gamma}}X_0=0$.   To do this, write   $X_0:=(X^1_0,X^2_0)$ and   $\gamma:=(\gamma^1,\gamma^2)$. Then,
\begin{equation}\label{vector-field-paral-12}
X_0^1(x)=\left \{
\begin{array}{ll}
0,&\text{ if } t_1=0,\\
\frac{b_xt_2^2-t_2(b_x-t_1)}{b_x^2+1},&\text{ if }
t_1\neq 0,
\end{array}
\right.
\quad  \mbox{and} \quad
X_0^2(x)=\left \{
\begin{array}{ll}
t_2,&\text{ if } t_1=0,\\
\frac{b_xt_2(b_x-t_1)+t_2^2}{b_x^2+1},&\text{ if }
t_1\neq 0.
\end{array}
\right.
\end{equation}
In expression of the differential equations (see, e.g., \cite[p. 53]{Carmo1992}), we only need to verify  that $X_0$ and $\gamma$ satisfy
\begin{equation}\label{differential equations-para}
\left \{
\begin{array}{l}
\frac{{\rm
d}(X_0^1\circ\gamma)}{{\rm
d}s}-\frac{X_0^1\circ\gamma}{\gamma^2}\frac{{\rm
d}\gamma^2}{{\rm
d}s}-\frac{X_0^2\circ\gamma}{\gamma^2}\frac{{\rm
d}\gamma^1}{{\rm
d}s}=0,\\
\frac{{\rm
d}(X_0^2\circ\gamma)}{{\rm
d}s}+\frac{X_0^1\circ\gamma}{\gamma^2}\frac{{\rm
d}\gamma^1}{{\rm
d}s}-\frac{X_0^2\circ\gamma}{\gamma^2}\frac{{\rm
d}\gamma^2}{{\rm
d}s}=0.
\end{array}
\right.
\end{equation}
Without loss of generality,  we   assume that $t_1\neq0$,
 and adopt the expression \eqref{geodesic} of the geodesic, that is
%
  $(\gamma^1(\cdot),\gamma^2(\cdot)) = (\gamma^1(b_x,r_x;\cdot),\gamma^2(b_x,r_x;\cdot))$ with
   \begin{equation}\label{vector-field-para100}
\gamma^1(b_x,r_x;s)=b_x-r_x\tanh s\quad\mbox{and}\quad \gamma^2(b_x,r_x;s)=\frac{r_x}{\cosh s}\quad \mbox{for any }s\in\IR,
\end{equation}
(noting $x_0=\gamma(b_x,r_x;{\rm artanh}\frac{b_x}{r_x})$ and $x=\gamma(b_x,r_x;{\rm artanh}\frac{b_x-t_1}{r_x})$),   where $b_x$ and $r_x$ are defined by \eqref{b-r-value}.
 Thus,  using \eqref{vector-field-para100}, one conclude that, for each $s\in\IR$,
\begin{equation}\label{vector-field-para2}
\begin{array}{ll}
&X_0^1\circ\gamma(b_x,r_x;s)=\frac{1}{b_x^2+1}(\frac{b_xr_x^2}{\cosh^2s}-\frac{r_x^2\sinh s}{\cosh^2s}),\\
&X_0^2\circ\gamma(b_x,r_x;s)=\frac{1}{b_x^2+1}(\frac{b_xr_x^2\sinh s}{\cosh^2s}+\frac{r_x^2}{\cosh^2s}),
\end{array}
\end{equation}
and so
\begin{equation}\label{vector-field-para3}
\begin{array}{ll}
&\frac{{\rm d}X_0^1\circ\gamma(b_x,r_x;s)}{{\rm
d}s}=\frac{1}{b_x^2+1}\left(-\frac{2b_xr_x^2\sinh s}{\cosh^3s}-\frac{r_x^2(1-\sinh^2s)}{\cosh^3s}\right),\\
&\frac{{\rm d}X_0^2\circ\gamma(b_x,r_x;s)}{{\rm
d}s}=\frac{1}{b_x^2+1}\left(\frac{b_xr_x^2(1-\sinh^2s)}{\cosh^3s}-\frac{2r_x^2\sinh s}{\cosh^3s}\right).
\end{array}
\end{equation}
Moreover,   we also have  that
\begin{equation}\label{vector-field-para1}
\frac{{\rm
d}\gamma^1(b_x,r_x;s)}{{\rm
d}s}=-\frac{r_x}{\cosh^2s} \quad\mbox{and} \quad\frac{{\rm
d}\gamma^2(b_x,r_x;s)}{{\rm
d}s}=-\frac{r_x\sinh s}{\cosh^2s}\quad \mbox{for any }s\in\IR.
\end{equation}
Thus, 
 \eqref{differential equations-para}  is seen to hold Hence  $\nabla_{\dot{\gamma}}X_0=0$, and  \eqref{vector-field-paral} is checked. 

Below we show the following assertions:

{\rm (i)} $f_0$ is not quasi-convex.

{\rm (ii)} ${\rm grad}f_0\neq X_0$.

{\rm (iii)} $\nabla_{\frac{\partial}{\partial t_1}}X_0\neq0$.

{\rm (iv)} $X_0$ is not a gradient vector field.


To show assertion (i),  take $x=(\frac{1}{2},\frac{1}{2}),\ y=(-\frac{1}{2},\frac{1}{2})\in\IH$, and let $c_0:=-0.4$. Then
 $x,y\in L_{c_0,f_0}$  because,   by \eqref{f-value} and \eqref{b-r-value},
$$f_0(x)=f_0(y)=\frac{1}{\sqrt{5}}({\rm artanh}\frac{2}{\sqrt{5}}-{\rm artanh}\frac{1}{\sqrt{5}})=-0.4304\cdots<-0.4.$$
Let $\gamma_{xy}$ be the geodesic segment joining $x$ to $y$. Then,
\begin{equation}\label{geodesic-y-z}
\gamma_{xy}(s):=\left(-\frac{1}{\sqrt{2}}\tanh\left((2s-1){\rm artanh}\frac{1}{\sqrt{2}}\right),\frac{1}{\sqrt{2}\cosh\left((2s-1){\rm artanh}\frac{1}{\sqrt{2}}\right)}\right) \quad \mbox{for any }s\in[0,1]
\end{equation}
thanks to   \eqref{geodesic2}, \eqref{geodesic-y-z-1} and \eqref{geodesic-y-z-2}. Hence
 $\gamma_{xy}(\frac{1}{2})=(0,\frac{1}{\sqrt{2}})$, and
$$f_0(\gamma_{xy}(\frac{1}{2}))=\ln\frac{1}{\sqrt{2}}=-0.3465\cdots>-0.4,$$
This means that
 $\gamma_{xy}(\frac{1}{2})\not\in L_{c,f_0}$, and so $L_{c,f_0}$ is not convex; see figure (3.1).  In view of Proposition \ref{remark}, we see that $f_0$ is not quasi-convex, and  assertion (i) holds.

\begin{figure}{}
\centering
\subfigure{\includegraphics[width=0.50\textwidth]{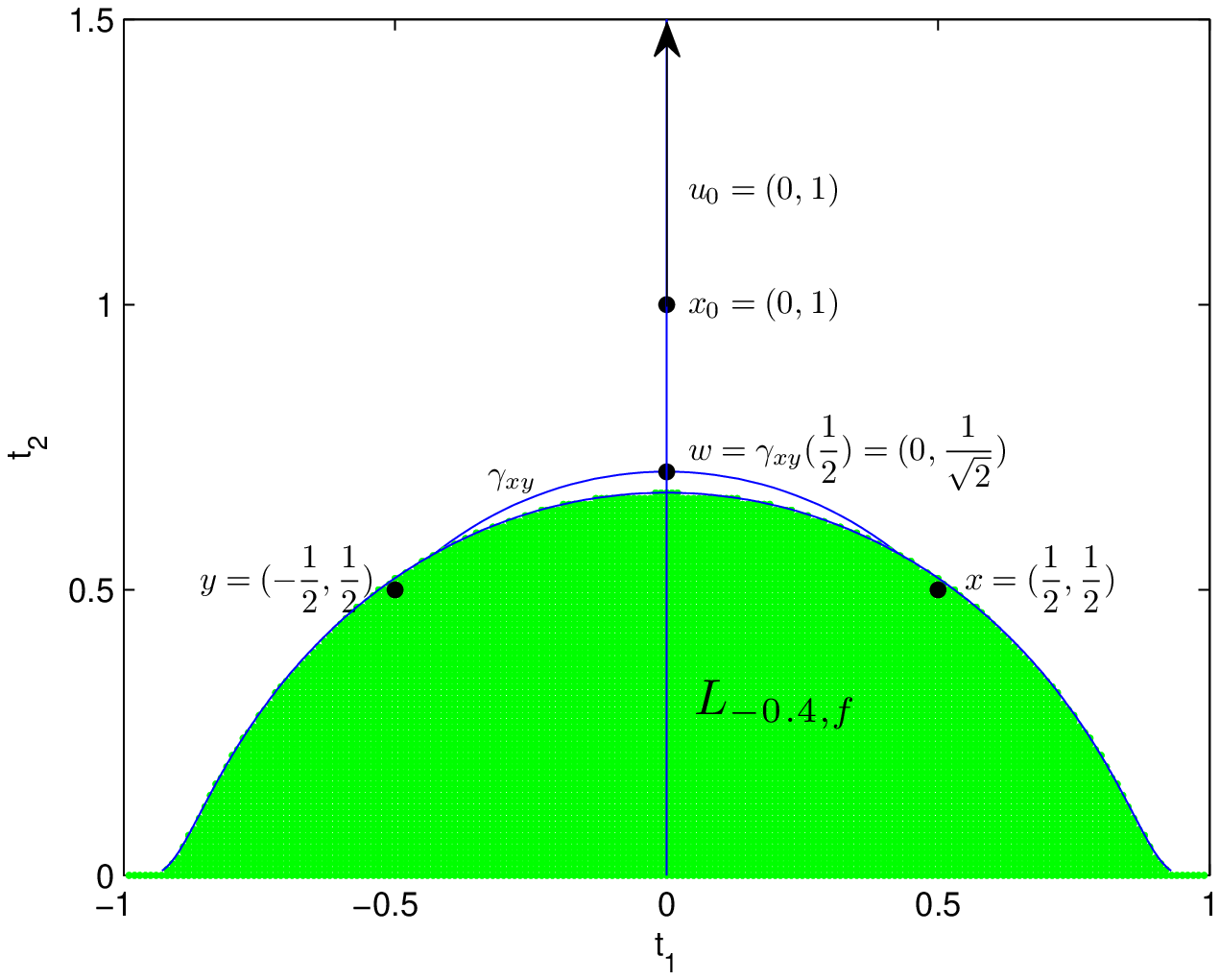}}
\caption*{Figure 3.1}
\end{figure}

To show assertion (ii), take $z:=(2,1)$. Then
\begin{equation}\label{z-value}
 b_{z}=1\quad\mbox{and}\quad r_z=\sqrt{2}
\end{equation}
 (see \eqref{b-r-value}). Therefore, we have by  \eqref{vector-field-paral} that
$X_0(z)=(1,0)$.
%
On the other hand, we get from \eqref{gradient-H} that
$${\rm grad}f_0(z)=\left(\frac{\partial f_0}{\partial t_1},\frac{\partial f_0}{\partial t_2}\right)$$
where $\frac{\partial f_0}{\partial t_1}$ and $\frac{\partial f_0}{\partial t_2}$ are classical partial derivatives in $\IR^2$. Then, using \eqref{f-value} and \eqref{b-r-value}, one calculates
$${\rm grad}f_0(z)=\left(\frac{\sqrt{2}}{8}\ln(3+2\sqrt{2})+\frac{1}{2},\frac{\sqrt{2}}{8}\ln(3+2\sqrt{2})-\frac{1}{2}\right).$$
Therefore ${\rm grad}f_0(z)\neq X_0(z)$,  and   assertion (ii) is checked. We further have that
\begin{equation}\label{h-n-x1}
\nabla_{\frac{\partial}{\partial t_1}}X_0(z)\neq0.
\end{equation}
Granting this,  assertion (iii) is also checked. To show \eqref{h-n-x1}, we get from \eqref{conenction-H-XY} that
\begin{equation}\label{h-n-1}
\nabla_{\frac{\partial}{\partial t_1}}X_0=\left(\frac{\partial X_0^1}{\partial t_1}-\frac{1}{t_2}X_0^2,\frac{\partial X_0^2}{\partial t_1}+\frac{1}{t_2}X_0^1\right).
\end{equation}
(noting that ${\frac{\partial}{\partial t_1}} =(1,0)$ for any $x\in\IH$).
Recalling that   $X_0^1$ and $X_0^2$ are given by \eqref{vector-field-paral-12} and $z:=(2,1)$,  we have that
$X_0^1(z)=1$ and  $X_0^2(z)=0$ (noting \eqref{z-value}). Furthermore,
  by elemental calculus, we can calculate the partial derivatives
\begin{equation}\label{h-n-3}
 \frac{\partial X_0^1}{\partial t_1}\mid_{z}=0\quad \mbox{and}\quad \frac{\partial X_0^2}{\partial t_1}\mid_{z}=-\frac{1}{2} .
\end{equation}
Thus we conclude from  \eqref{h-n-1}  that $\nabla_{\frac{\partial}{\partial t_1}}X_0\mid_{z}=(0,\frac{1}{2})\neq0$, as desired to show.

For assertion (iv), we suppose on the contrary that there exists a $C^\infty$ function $f$ such that $X_0={\rm grad}f$. Then
${\rm d}\circ {\rm d}f =0$ by the fundamental property (see, e.g., \cite[p. 17]{Sakai1996}). To proceed, note that
 $X_0=X_0^1\frac{\partial}{\partial t_1}+X_0^2\frac{\partial}{\partial t_2}$, where $X_0^1$ and $X_0^2$ are defined by \eqref{vector-field-paral-12}.
Then, we calculate by elementary calculus that
\begin{equation}\label{dd-z}
\left.\left(\frac{\partial(\frac{1}{t_2^2}X_0^1)}{\partial t_2}-\frac{\partial(\frac{1}{t_2^2}X_0^2)}{\partial t_1}\right)\right|_{x=(2,1)}=\frac{1}{2}\neq0.
\end{equation}
Furthermore, by \eqref{gradient-H} and \eqref{differential-H}, one has that
$${\rm d}f=\frac{1}{t_2^2}X_0^1{\rm d}t_1+\frac{1}{t_2^2}X_0^2{\rm d}t_2,$$
 and so the exterior differentiation
$${\rm
d}\circ {\rm
d}f=\left(\frac{\partial(\frac{1}{t_2^2}X_0^2)}{\partial t_1}-\frac{\partial(\frac{1}{t_2^2}X_0^1)}{\partial t_2}\right){\rm d}t_1\wedge {\rm d}t_2,$$
where $\wedge$ is the exterior product; see, e.g., \cite[p. 17]{Sakai1996}. This, together with \eqref{dd-z}, means that ${\rm d}\circ {\rm d}f\not =0$, and so
%
assertion (iv) is shown.

}

\end{example}

\section{Convexity properties of sub-level sets on Riemannian manifolds}
\label{sec:3}
Throughout this section, let $\kappa\in\IR$ and  assume that $M$ is a complete,
simply connected Riemannian manifold of constant sectional curvature
$\kappa$. As usual, define $D_\kappa:=\frac{\pi}{\sqrt{\kappa}}$ if $\kappa>0$ and $D_\kappa:=+\infty$ otherwise.
 Then, for any point $x,y\in M$ with ${\rm d}(x,y)<D_\kappa$, $\Gamma_{xy}$ contains a unique   minimal geodesic, (which will be  denoted by $\gamma_{xy}$), and
  any open ball $\IB(x,r)$ with $r\le \frac{D_\kappa}{2}$ is strongly convex for any $x\in M$;  see e.g., \cite[Proposition 4.1 (i)]{LiY2012}.   Let $x_0\in M$ and $u_0\in T_{x_0}M\setminus\{0\}$. Consider the following function $f_0:M\rightarrow\overline {\IR}$ defined by
\begin{equation}\label{f2}
f_0(x)=\left \{
\begin{array}{ll}
\langle u_0,\dot{\gamma}_{x_0x}(0)\rangle,&\text{ if } x\in \IB(x_0,\frac{D_\kappa}{2}),\\
+\infty,&\text{ otherwise},
\end{array}
\right.
\end{equation}
where ${\gamma}_{x_0x}(0)\in{\Gamma}_{x_0x}$ is the unique minimal geodesic lying in $\IB(x_0,\frac{D_\kappa}{2})$. It is clear that ${\rm dom}f_0=\IB(x_0,\frac{D_\kappa}{2})$ is strongly convex. If $M$ is a Hadamard manifold, function \eqref{f2} is reduced  to the function defined by \eqref{form}, that is
\begin{equation}\label{f3}
f_0(x):=\langle u_0,\exp_{x_0}^{-1}x\rangle\quad\mbox{for any $x\in M$}.
\end{equation}
For any $c\in\IR$,  the sub-level set of $f_0$ is denoted by $L_{c,f_0}(c\in\IR)$ and defined by
$$
L_{c,f_0}:=\{x\in M:f_0(x)\le c\}.
$$

Note by Example \ref{example} that $L_{c,f_0}$ is not strongly convex in general. This section is devoted to
 study of the   convexity property of the sub-level sets $L_{c,f_0}$ ($c\in\IR$). For this purpose, we first
 recall that a geodesic triangle $\triangle(p_{1}p_{2}p_{3})$ in $M$
is a figure consisting of three points $p_{1},p_{2},p_{3}$ (the
vertices of $\triangle(p_{1}p_{2}p_{3})$) and three minimal geodesic
segments $\gamma_{i}$ (the edges of $\triangle(p_{1}p_{2}p_{3})$)
such that $\gamma_i(0)=p_{i-1}$ and $\gamma_i(1)=p_{i+1}$ with
$i=1,2,3$ ( mod$3$). For each $i=1,2,3$ ( mod$3$), the inner angle
of $\triangle(p_{1}p_{2}p_{3})$ at $p_{i}$ is denoted by $\angle
p_{i}$, which equals the angle between the tangent vectors
$\dot{\gamma}_{i+1}(0)$ and $-\dot{\gamma}_{i-1}(1)$. The following proposition (i.e., comparison theorem for triangles) follows immediately from \cite[p.161 Theorem 4.2 (ii), p. 138 Low of Cosines and p.  167 Remark 4.6]{Sakai1996}.

\begin{proposition}\label{compare}
Let $\triangle(p_{1}p_{2}p_{3})$ be a geodesic
triangle in $M$ of the perimeter less than
$2D_{\kappa}$.
Set $l_i={\rm d}(p_{i+1},p_{i-1})$ for each $i=1,2,3$. Then, the
following relations hold:
 \begin{equation}\label{cop>0}
 l_i^2< l_{i-1}^2+l_{i+1}^2-2l_{i-1}l_{i+1}\cos\angle p_i\quad\mbox{if $\kappa>0$},
\end{equation}
and
\begin{equation}\label{cop>0}
 l_i^2> l_{i-1}^2+l_{i+1}^2-2l_{i-1}l_{i+1}\cos\angle p_i\quad\mbox{if $\kappa<0$}.
\end{equation}

\end{proposition}

 Another property for  Riemannian manifolds of constant curvature, which will be used in sequel,  is   the     axiom of plane  described as follows (see, e.g., \cite[p. 136]{Sakai1996}):
 \begin{proposition}\label{axiom}
 Let  $x\in M$ and let $W$ be a $k$-dimensional subspace of $T_xM$. Then the submanifold $N:=\exp_x(W\cap \IB(0_x,\rho))$ is a $k$-dimensional totally geodesic submanifold of $M$ for any  $0<\rho<D_\kappa$. Recall  a $k$-dimensional submanifold $N\subset M$  is  totally geodesic  iff any geodesic $\gamma$ of $M$ with the initial direction $u\in TN$ is contained in $N$; see, e.g., \cite[p. 48]{Sakai1996}.
\end{proposition}
 The
following lemma, taken from \cite[Theorem 3.1 and Remark 3.6]{Afsari2013},   plays a very key role in our study afterwards.

\begin{lemma}\label{cmpr1}
Let   $\triangle(ypq)$ be a geodesic triangle in $M$
of the perimeter less than $2D_\kappa$. Let
$\triangle(\tilde y\tilde p\tilde q)$ be a triangle in $\IR^2$ such
that
\begin{equation}\label{tri-comp}
  {\rm d}(y,p)=\|\overrightarrow{\tilde y\tilde p}\|,\quad {\rm d}(y,q)=\|\overrightarrow{\tilde y\tilde q}\|\quad\mbox{and}\quad \angle pyq=\angle \tilde p\tilde y\tilde
q.
\end{equation}
Let   $x$ be in the minimal geodesic joining $p$ to $q$, and
$\tilde x$ be the corresponding point  in the interval $[\tilde p,\tilde
q]$ satisfying
\begin{equation}\label{crro-pint}
 \angle pyx=\angle\tilde p\tilde y\tilde x \quad\mbox{and}\quad
 \angle qyx=\angle\tilde q\tilde y\tilde x
\end{equation}
 (see Figure 4.1). Then, the following assertions hold:
\begin{equation}\label{distance0}
 {\rm d}(y,x)\ge\|\overrightarrow{\tilde y\tilde x}\|\;\mbox{ if $\kappa\ge0$}\quad\mbox{and}\quad
 {\rm d}(y,x)\le\|\overrightarrow{\tilde y\tilde x}\|\;\mbox{ if $\kappa\le0$}.
\end{equation}
\end{lemma}

\begin{figure}\label{}
\centering
\subfigure{\includegraphics[width=0.7\textwidth]{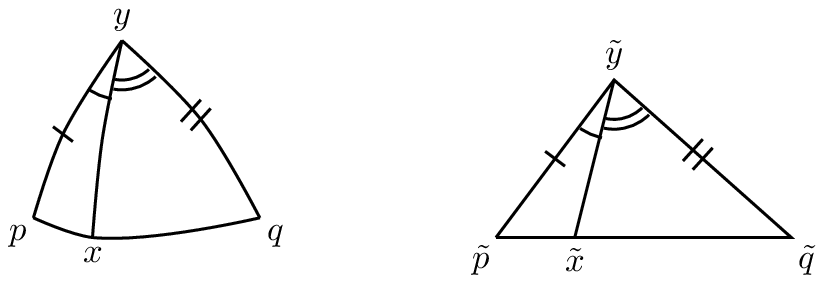}}
\caption*{Figure 4.1}
\end{figure}

Recall that, for any $x,y\in M$,  $\gamma_{xy}\in \Gamma_{xy}$ denote the  unique   minimal geodesic:   $\gamma_{xy}:[0,1]\to M$ is the minimal geodesic satisfying $\gamma_{xy}(0)=x$ and  $\gamma_{xy}(1)=y$.

\begin{lemma}\label{cmpr2}
Let  $\triangle(ypq)$ be a geodesic triangle in $M$
of the perimeter less than $2D_\kappa$. Let
$\gamma:=\gamma_{pq}:[0,1]\rightarrow M$ be the unique minimal geodesic joining $p$ to $q$.  Then, for each
$t\in(0,1)$, there exist two positive numbers $a_t$ and $b_t$ satisfying
\begin{equation}\label{plus=ab}
   a_t+b_t\;\left\{\begin{array}{ll}
   \ge1,\quad &\mbox{if } \kappa\ge0,\\
   \le1, &\mbox{if } \kappa\le0,\\
   \end{array}
   \right.
\end{equation}
 such that
\begin{equation}\label{cmpr3}
\dot{\gamma}_{y\gamma(t)}(0)=a_t\dot{\gamma}_{yp}(0)+b_t\dot{\gamma}_{yq}(0).
\end{equation}
\end{lemma}

\begin{proof}
Since the perimeter of the geodesic triangle $\triangle(ypq)$ is less than $2D_\kappa$, one can verify  that $\rho=\max\{\|\dot{\gamma}_{yp}(0)\|,\|\dot{\gamma}_{yq}(0)\|\}<D_\kappa$. Let $\rho<\bar\rho<D_\kappa$. Then, we get from Proposition \ref{axiom} that $N:=\exp_y\{{\rm span}\{\dot{\gamma}_{yp}(0),\dot{\gamma}_{yq}(0)\}\cap\IB(0_y,\bar\rho)\}$ is 2-dimensional totally geodesic sub-manifold of $M$. Hence   $\gamma\subset N$ thanks to assumption.  Thus, one has that
\begin{equation}\label{lc1}
\dot{\gamma}_{y\gamma(t)}(0)\in T_{y}N\subseteq{\rm span}\{\dot{\gamma}_{yp}(0),\dot{\gamma}_{yq}(0)\}\quad\mbox{for any } t\in[0,1].
\end{equation}
Thus, there exist some $a_t,b_t\in \IR$ such that \eqref{cmpr3} holds (see figure 4.1).

Below, we show that $a_t,b_t$ are positive and satisfy \eqref{plus=ab}. To this end,  as  in Lemma \ref{cmpr1} (see Figure 1), set $x=\gamma(t)$,  and let $\triangle(\tilde y\tilde p\tilde
q)$  be the corresponding  triangle of $\triangle(ypq)$ in $\IR^2$ satisfying \eqref{tri-comp} and $\tilde x$ be the corresponding point   in the interval $[\tilde p,\tilde q]$ satisfying \eqref{crro-pint}.
Without loss of generality, we may assume by \eqref{tri-comp} that $\overrightarrow{\tilde y\tilde p}=\dot{\gamma}_{yp}(0)$ and $\overrightarrow{\tilde y\tilde q}=\dot{\gamma}_{yq}(0)$.
Note, by \eqref{lc1}, that the vectors $\overrightarrow{\tilde y\tilde x}$ and $\dot{\gamma}_{yx}(0)$ are in the same $2$-dimensional Euclidean plane.
It follows from \eqref{crro-pint}, together with \eqref{cmpr3},  that
 there exists some $\lambda>0$ such that
\begin{equation}\label{cmpr32}
\lambda\overrightarrow{\tilde y\tilde x}=\dot{\gamma}_{yx}(0)=a_t\dot{\gamma}_{yp}(0)+b_t\dot{\gamma}_{yq}(0).
\end{equation}
Note that   $\tilde x$ lies actually  in the open interval $(\tilde p,\tilde
q)$ in $\IR^2$ (as
$0<t<1$ and so  $  \angle\tilde p\tilde y\tilde x>0,\;
 \angle\tilde q\tilde y\tilde x>0$  by \eqref{crro-pint}).
It follows from \eqref{cmpr32} that
\begin{equation}\label{cmpr33}
 a_t>0,\quad b_t>\quad \mbox{and}\quad \frac{a_t+b_t}\lambda=1.
\end{equation}
 Furthermore,
in view of \eqref{distance0}, we see that
$\lambda\le1$ if $ \kappa\ge0$   and $\lambda\ge1$ if $\kappa\le0$.
%
This, together with \eqref{cmpr33}, implies that \eqref{plus=ab} holds and the proof is complete.
\end{proof}




Now we are ready to verify the first theorem in the present section.

\begin{theorem} \label{levelset1} Suppose that the constant sectional curvature $\kappa>0$ and let $f_0$ be the function defined by \eqref{f2}. Then
the sub-level set   $L_{c,f_0}$ is strongly convex if and only if either  $c\le0$ or  $c\ge \frac{\|u_0\|D_\kappa}{2}$.
\end{theorem}

\begin{proof}
We first show the sufficiency part. To do this, suppose that $c\le0$ or  $c\ge \frac{\|u_0\|D_\kappa}{2}$.
 Note that if $c\ge \frac{\|u_0\|D_\kappa}{2}$ then $L_{c,f_0}=\IB(x_0,\frac{D_\kappa}{2})$ is strongly convex because
 $$f_0(x)=\langle u_0,\dot{\gamma}_{x_0x}(0)\rangle\le\|u_0\|\cdot\|\dot{\gamma}_{x_0x}(0)\|\le\frac{\|u_0\|D_\kappa}{2}\le c.
$$
holds  for all $x\in \IB(x_0,\frac{D_\kappa}{2})$.
Thus, we need only to consider the case when $c\le0$. To proceed, fix $c\le0$ and let $p,q\in L_{c,f_0}$, that is, 
\begin{equation}\label{pqc}
\langle u_0,\dot{\gamma}_{x_0p}(0)\rangle\le c\quad\mbox{and}\quad \langle u_0,\dot{\gamma}_{x_0q}(0)\rangle\le c.
\end{equation}
Then  $p,q\in \IB(x_0,\frac{D_\kappa}{2})$ and  the geodesic triangle $\triangle(x_0pq)$ is well defined with perimeter less than $2D_\kappa$. 
Let $t\in [0,1]$. By assumption, Lemma
\ref{cmpr2} is applicable to concluding  that there exist two positive numbers $a_t$ and $b_t$ satisfying
with $a_t+b_t\ge1$ such that
$$\dot{\gamma}_{x_0\gamma(t)}(0)=a_t\dot{\gamma}_{x_0p}(0)+b_t\dot{\gamma}_{x_0q}(0),$$
where $\gamma:=\gamma_{pq}:[0,1]\rightarrow M$ is the unique minimal
 geodesic joining $p$ and $q$.
 It follows from  \eqref{f2} and \eqref{pqc} that
$$ f_0(\gamma(t)) =a_t\langle u_0,\dot{\gamma}_{x_0p}(0)\rangle+b_t\langle u_0,\dot{\gamma}_{x_0q}(0)\rangle
\le c(a_t+b_t)\le c $$
(note that $c<0$). This means that $\gamma_{p,q}(t)=\gamma(t)\in L_{c,f_0}$ for all $t\in [0,1]$, and so
$L_{c,f_0}$ is strongly convex as desired to show. The proof for the sufficiency part is complete.

To show the necessity part, without loss of generality, we may assume that
 $\|u_0\|=1$.  Let $0<c<\frac{D_\kappa}{2}$. It suffices to verify that  $L_{c,f_0}$ is not strongly convex, or equivalently, to construct two  points $p,\,q$ and a number $\bar t\in (0,1)$ such that
 \begin{equation}\label{non-in}
   p, q\in L_{c,f_0}\quad\mbox{and}\quad \bar z:=\gamma_{pq}(\bar t)\notin L_{c,f_0}.
 \end{equation}
 To do this,
consider the geodesic $\gamma:[0, \frac{D_\kappa}{2})\to M$ defined by $\gamma(t):=\exp_{x_0}tu_0$ for each  $t\in[0,\frac{D_\kappa}{2})$. Clearly it is contained in $\IB(x_0,\frac{D_\kappa}{2})$. Since $\IB(x_0,\frac{D_\kappa}{2})$ is strongly convex, we see that, for each $ t\in [0, \frac{D_\kappa}{2})$,
 the unique minimal geodesic joining $x_0$ and $\gamma(t)$ can be expressed as
 $$\gamma_{x_0\gamma(t)}(s)=\exp_{x_0}s(tu_0) \quad\mbox{for each } s\in [0,1].$$
 This in particular implies that, for each $ t\in [0, \frac{D_\kappa}{2})$, $\dot{\gamma}_{x_0\gamma(t)}(0)=tu_0$  and  so
\begin{equation}\label{nqsic1-v}
f_0(\gamma(t))=\langle u_0,\dot{\gamma}_{x_0\gamma(t)}(0)\rangle=\langle u_0,tu_0\rangle=t.
\end{equation}
Hence
 \begin{equation}\label{nqsic1}
\gamma(t) \in L_{c,f_0} \mbox{ for all }  t\in [0,c]\quad\mbox{and}\quad  \gamma(t)\not\in L_{c,f_0} \mbox{ for all } t\in (c,\frac{D_\kappa}{2})
\end{equation}
because
\begin{equation}\label{ball-z1}
{\rm d}({x_0},z)=c<\frac{D_\kappa}{2},
\end{equation}
by the choice of  $c$. In particular,
 $z:=\gamma(c)  \in L_{c,f_0}$.
Take $u\in T_zM$ such that $u\perp\dot{\gamma}(c)$. Then, by \eqref{ball-z1}, there exists some $\varepsilon>0$
such that the geodesic
$\tau:[-\varepsilon,\varepsilon]\rightarrow M$, determined by $\tau(0)=z$ and  $\dot{\tau}(0)=u$, is contained in
$\IB(x_0,\frac{D_\kappa}{2})\cap\IB(z,\frac{D_\kappa}{2})$. Set $p_{\varepsilon}:=\tau(\varepsilon)$ and  $q_{\varepsilon}:=\tau(-\varepsilon)$ (see, Figure 4.2). Then
\begin{equation}\label{ball-z2}
p_{\varepsilon},\;q_{\varepsilon}\in\IB(x_0,\frac{D_\kappa}{2})\cap\IB(z,\frac{D_\kappa}{2}).
\end{equation}
 Below, we shall show that
\begin{equation}\label{nc11}
p_\varepsilon,\;q_\varepsilon\in L_{c,f_0}\mbox{ with $f_0(p_\varepsilon)<c$ and $f_0(q_\varepsilon)<c$}.
\end{equation}
Consider the geodesic triangle
$\triangle(x_0zp_\varepsilon)$. Then its perimeter is less than $2D_\kappa$ thanks to \eqref{ball-z1} and  \eqref{ball-z2}. Thus Proposition \ref{compare} is applicable, and using \eqref{cop>0}, we have that
$${\rm d}^2(x_0,p_\varepsilon)<{\rm d}^2(x_0,z)+{\rm d}^2(z,p_\varepsilon)-2{\rm d}(x_0,z){\rm d}(z,p_\varepsilon)\cos \angle p_\varepsilon zx_0={\rm d}^2(x_0,z)+{\rm d}^2(z,p_\varepsilon), $$
(noting that $\angle p_\varepsilon zx_0=\frac\pi2$ as $\dot{\tau}(0)\perp\dot{\gamma}(c)$), and
$${\rm d}^2(z,p_\varepsilon)<{\rm d}^2(x_0,z)+{\rm d}^2(x_0,p_\varepsilon)-2{\rm d}(x_0,z){\rm d}(x_0,p_\varepsilon)\cos\angle
p_\varepsilon x_0z.$$
Combining these two inequalities, we get that
$${\rm d}(x_0,p_\varepsilon)\cos\angle
p_\varepsilon x_0z<{\rm d}(x_0,z).$$
Thus
$$f_0(p_\varepsilon)={\rm d}(x_0,p_\varepsilon)\cdot\|u_0\|\cdot\cos\angle p_\varepsilon x_0z={\rm d}(x_0,p_\varepsilon)\cos\angle p_\varepsilon x_0z<{\rm d}(x_0,z)=c,$$
where the last equality holds because of \eqref{ball-z1}.    Similarly,   we have $f_0(q_\varepsilon)<c$ and  \eqref{nc11} is shown.

Let $\gamma_{x_0}:[0,\infty)\rightarrow M$ be the geodesic satisfying that
$\gamma_{x_0}(0)=x_0$ and $\gamma_{x_0}(1)=p_\varepsilon$. In light of \eqref{ball-z2} and \eqref{nc11}, we get by the continuity of $f_0$ that
there exists $t_0>1$ such that $\gamma_{x_0}(t_0)\in L_{c,f_0}$. Set
$p :=\gamma_{x_0}(t_0)$ and $q:=q_{\varepsilon}$. Then,  $p,\; q\in L_{c,f_0}$
(see \eqref{nc11}). We further show that
 \begin{equation}\label{nc12}
  \bar z:=\gamma_{pq}(\bar t)\notin L_{c,f_0} \;\mbox{ for some }\bar t\in (0,1).
\end{equation}
Granting this, \eqref{non-in} is established. To show \eqref{nc12}, write
 $\widetilde{N}:=\exp_z\{{\rm span}\{\dot{\gamma}_{zx_0}(0),u\}\cap\IB(0_z,\frac{D_\kappa}{2})\}$. Then $\widetilde{N}$ is a 2-dimensional totally geodesic sub-manifold of $M$ by Proposition \ref{axiom} (recalling that $M$ is of constant curvature). Since points $x_0,p,q,p_\varepsilon,z$ lie in $\widetilde{N}$, it follows that  $\gamma_{pq}$ must meet $\gamma $ at some point $\bar z:=\gamma_{pq}(\bar t)=\gamma(c_0)$ with $\bar t\in (0,1)$ and $c_0>c$ (see Figure 4.2). In light of \eqref{nqsic1}, one sees  that $\bar z\not\in
L_{c,f_0}$. Thus \eqref{nc12} is shown,  and the proof is complete.
\end{proof}

\begin{figure}{}
\centering
\subfigure{\includegraphics[width=0.5\textwidth]{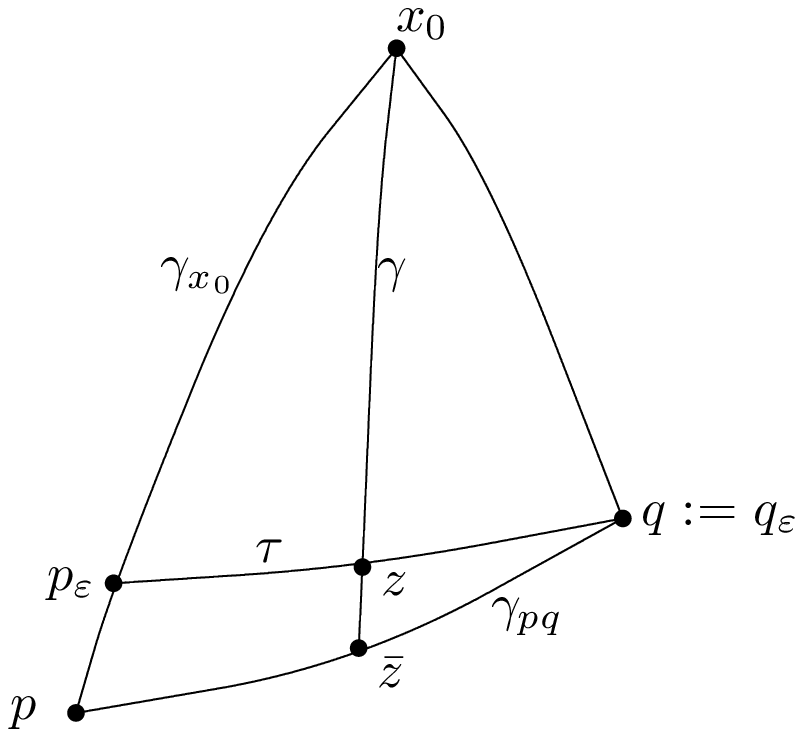}}
\caption*{Figure 4.2}
\end{figure}

Our second theorem in this section is  Theorem \ref{levelset2} below, which is an analogue of  Theorem \ref{levelset1} on Hadamard manifold of constant
sectional curvature.  In particular,  Theorem \ref{levelset2} improves and extends  the corresponding result in \cite[Corollary 3.1]{Ferreira2005}, where it was shown that  the sub-level sets $L_{c,f_0}$ is convex  in the special case when $c=0$. The proof of Theorem \ref{levelset2} is quite similar to  that we did for  Theorem \ref{levelset1} and so  we omit it here.

\begin{theorem} \label{levelset2}
Suppose that the constant sectional curvature $\kappa<0$ and let $f_0$ be the function defined  by \eqref{f3}.
Then,  $L_{c,f_0}$ is convex  if and only if $c\ge0$.
\end{theorem}

As a direct consequence of Theorems \ref{levelset1} and \ref{levelset2}, together with  Proposition \ref{remark}, we have the following corollary which shows that the function defined by \eqref{f2} is not quasi-convex in general.

\begin{corollary}\label{corollary}
Suppose that  $M$ is of non-zero constant sectional
curvature. Let $x_0\in M$ and $u_0\in T_{x_0}M\setminus \{0\}$. Then, the  functions defined by \eqref{f2} is not quasi-convex.
\end{corollary}


\end{CJK*}

\end{document}